\documentclass[preprint,12pt]{elsarticle}

\usepackage{amssymb}
\usepackage{amsmath}
\usepackage{amsthm}
\usepackage{arydshln}
\usepackage{graphicx}
\usepackage{bm}
\usepackage{subfigure}
\usepackage{multirow}
\usepackage{booktabs}
\usepackage{rotating}

\newtheorem{theorem}{Theorem}
\newtheorem{remark}{Remark}

\journal{ } 

\begin{document}
\bibliographystyle{elsarticle-num}
\newcommand{\modN}{\ (\mbox{mod}\ N)}

\begin{frontmatter}



\title{Phase models and clustering in networks of oscillators with delayed coupling}


\author[label1]{Sue Ann Campbell}
\author[label2]{Zhen Wang}
\address[label1]{Department of Applied Mathematics and Centre for Theoretical Neuroscience, University of Waterloo
   Waterloo ON N2L 3G1 Canada (e-mail: sacampbell@uwaterloo.ca).}
\address[label2]{Department of Applied Mathematics, University of Waterloo
   Waterloo ON N2L 3G1 Canada (e-mail: z377wang@uwaterloo.ca).}

\begin{abstract}
We consider a general model for a network of oscillators with 
time delayed, circulant coupling. We use the theory of weakly
coupled oscillators to reduce the system of delay differential 
equations to a phase model where the time delay enters as a phase 
shift. We use the phase model to study the existence and stability
of cluster solutions. Cluster solutions are phase locked
solutions where the oscillators separate into groups. Oscillators
within a group are synchronized while those in different groups
are phase-locked. We give model independent existence and stability
results for symmetric cluster solutions. We show that the presence
of the time delay can lead to the coexistence of multiple stable
clustering solutions. We apply our analytical results to a 
network of Morris Lecar neurons and compare these results with
numerical continuation and simulation studies.
\end{abstract}

\begin{keyword}
Time delay \sep neural network \sep oscillators \sep clustering solutions \sep stability



\end{keyword}

\end{frontmatter}


\section{Introduction}
\label{Introduction}
Coupled oscillator models have been used to study many biological
and physical systems, for example neural networks \citep{HMM93,KE88},
laser arrays \citep{WWa,WWb}, flashing of fireflies \citep{MS}, and
movement of a slime mold \citep{TFE}.  A basic question explored with
such models is whether the elements in the system will {\bf phase-lock}, i.e.,
oscillate with some fixed phase difference, and how the
physical parameters affect the answer to this question.
Clustering is a type of phase locking behavior where the oscillators in a network separate into
groups. Each group consists of fully synchronized oscillators, and different groups are
phase-locked with nonzero phase difference. Symmetric clustering refers to the situation 
when all the groups
are the same size while non-symmetric clustering means the groups have different sizes.

A phase model represents each oscillator with a single
variable as shown in Figure~\ref{phasedef}.
\begin{figure}
\centerline{\includegraphics[width=4in]{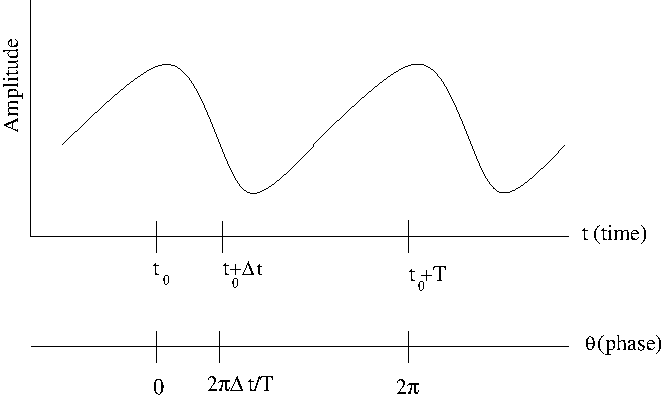}}
\caption{Defining the phase of an oscillator.}
\label{phasedef}
\end{figure}
A differential equation for each phase variable indicates how the phase
of the oscillator changes in time:
\[ \frac{d\theta_i}{dt}=\Omega_i+H_i(\theta_1,\theta_2,\ldots,\theta_N) \]
Here $\Omega_i$ is the intrinsic frequency of the $i^{th}$ oscillator and
the functions $H_i$ described how the coupling between oscillators influences
the phases.  Phase models have been used to study the behaviour of networks of
coupled oscillators beginning with the work of \cite{K84}.
Phase models are sometimes {\em posed} as models for coupled oscillators
\cite{MS,K84,Okuda93,Li03}. When the coupling between oscillators is sufficiently
weak, however, a phase model representation of a system can be {\em derived}
from a higher dimensional differential equation model, such one obtained from
a physical or biological description of the system \cite{ET10,HI97,KE02,SL12}. 
The low dimensional phase
model can then be used to predict behaviour in original high dimensional
physical model. This approach
has proved useful in studying synchronization properties of
many different neural models
\cite{HMM93,BC99a,Crook97,Erm96,galan2009phase,HMM95,MLPR07,ZS09}.
Phase models can be linked to experimentally derived phase
resetting curves \cite{ET10,SL12}, thus this approach has also been used to
make predictions about synchronization properties of experimental
preparations \cite{MLPR07}.

Okuda \cite{Okuda93} was the first to use phase models to study clustering behaviour.
Considering a phase model for a network of arbitrary size with all-to-all coupling,
Okuda \cite{Okuda93}
established general criteria for the stability of all possible symmetric cluster solutions as well
as some non-symmetric cluster solutions. He showed that these results gave a good prediction
of stability for a variety of model networks. Recently, similar results have been obtained for networks
with nearest-neighbour coupling \citep{MRTWBC}. Phase model analysis has been extensively used to study phase-locking in pairs of model and experimental neurons \citep{KE02,SNS06,MLPR07}. More recently it has been used to study clustering in larger
neural networks \cite{KE11,galan2006}.

In many systems there are time delays in the connections between the oscillators due to the time for a signal to propagate
from  one element to the other. In neural networks this delay is attributed
to the conduction of electrical activity along an axon or a dendrite
\cite{Crook97,KE02}.
Much work has been devoted to the study of the effect of time delays in
neural networks. However, the majority of this work has focussed on
systems where the neurons are excitable not oscillatory,
(e.g., \cite{BT,BT2,BGV,DHPS09,PRHS13,SHHD09}), the networks have only a few neurons
(e.g., \cite{Li03,CK12,KE02,LWM03,SW}) or focussed exclusively on synchronization (e.g., \cite{Crook97,Luz,Orosz12,Orosz14,PRHS13}).  Extensive work has been done
on networks Stuart-Landau oscillators with delayed diffusive coupling
(e.g., \cite{choe2010,CDHS10,DLS12} where the model for the individual oscillators
is the normal form for a Hopf bifurcation and thus the system is often
amenable to theoretical analysis.  Numerical approaches to study the
stability of cluster solutions in delayed neural oscillator networks
have also been developed \cite{Orosz14,OroszSIADS14}.
We note that there is a vast literature on time delays in artificial neural networks which we do not attempt to cite here.

Initial studies of phase models for systems with delayed coupling
considered models where the delay occurs in the argument of
the phases \cite{SW,Luz,KPR,NSK,Strog99}.
However, it has been shown \cite{KE02,Erm94,Izhikevich98} that for small
enough time delays it is more appropriate to include the
time delay as phase shift in the argument of the coupling function.
Crook et al. \cite{Crook97} use this type of model to study a continuum of
cortical oscillators with spatially decaying coupling and axonal delay.
Bressloff and Coombes \cite{BC99a,BC99b} study phase locking in chains and rings of pulse
coupled neurons with distributed delays and show
that distributed delays result in phase models with a distribution
of phase shifts. They consider phase models derived from integrate
and fire neurons and the Kuramoto phase model.

In this paper, we use phase models to investigate the effect of time delayed
coupling on the clustering behavior of oscillator networks. The plan for
our article as as follows. In the next section we will review how a general
network model with delayed coupling may be reduced to a phase model.  In section
3 we give conditions for existence and stability of symmetric
cluster solutions in a network with a circulant coupling matrix. In section
4 we
consider a particular application: a network of Morris-Lecar oscillators. We
derive the particular
phase model for this system and compare the predictions of the phase model theory to
numerical continuation and simulation studies. In section 5 we consider networks
where the connection matrix is no longer circulant.  In section 6 we discuss our work.

\section{Reduction to a phase model}
\label{section reduction to a phase model}

In this section, we review how to reduce a general model for a network of all-to-all coupled oscillators with time-delayed connections to a phase model. We begin by considering
our model for a single oscillator. This is a $n$-dimensional system of ordinary
differential equations
\begin{eqnarray}\label{PM ODE}
\frac{dX}{dt}=F(X(t)),
\end{eqnarray}
which admits an exponentially asymptotically stable periodic orbit, denoted by $\hat{X}(t)$, with period $T$. Linearizing the model (\ref{PM ODE}) about the periodic solution $\hat{X}(t)$ we obtain
\begin{eqnarray}\label{PM ODE linear}
\frac{dX}{dt}=DF(\hat{X}(t))X,
\end{eqnarray}
and its adjoint system
\begin{eqnarray}\label{PM ODE linear adjoint}
\frac{dZ}{dt}=-[DF(\hat{X}(t))]^TZ.
\end{eqnarray}
Here $DF(\hat{X}(t))$ represents the Jacobian matrix of F with respect to $X$, evaluated
on the periodic orbit $\hat{X}(t)$. Denote by $Z=\hat{Z}(t)$ the unique periodic solution of the adjoint system (\ref{PM ODE linear adjoint}) satisfying the normalization condition:
\begin{eqnarray*}
\frac{1}{T}\int_{0}^T\hat{Z}(t)\cdot F(\hat{X}(t))dt =1.
\end{eqnarray*}

Now, consider the following network of identical oscillators with all-to-all, time-delayed coupling
\begin{equation}\label{PM general network DDE}
\frac{dX_i}{dt}=F(X_i(t))+\epsilon\!\! \sum_{j=1}^N w_{ij} G(X_i(t),X_j(t-\tau_{ij})),\ i=1, \cdots, N.
\end{equation}
Here $G:\mathbb{R}^n\times\mathbb{R}^n\rightarrow\mathbb{R}^n$ describes the coupling between two oscillators,
$\epsilon$ is referred to as the coupling strength, and $W=[w_{ij}]$ is the coupling
matrix.

When $\epsilon$ is sufficiently small and $w_{ij}$ are of order $1$ with respect to $\epsilon$, we can apply the theory of weakly coupled oscillators to reduce (\ref{PM general network DDE}) to a phase model \cite{ET10,HI97,KE02}. The ways in which the time delay enters
into the phase model depends on the size of the delay relative to other time constants in the model.  Let $\Omega=2\pi/T$.
It has been shown \cite{KE02,Erm94,Izhikevich98} that if the delays
satisfy $\Omega\tau_{ij} =\mathit{O}(1)$ with respect to the coupling strength $\epsilon$,
then the appropriate model is
\begin{equation}\label{DDE phase model 0}
\frac{d\theta_i}{dt}=\Omega+\epsilon \sum_{j=1}^N W_{ij} H(\theta_j-\theta_i-\eta_{ij})+\mathit{O}(\epsilon^2), \ i=1, 2, \cdots, N,
\end{equation}
where $\eta_{ij}=\Omega \tau_{ij}$.  That is, the delays enter as phase lags.
The interaction function $H$ is a $2\pi$-periodic function which satisfies
\[
H(\theta)=\frac{1}{T}\int_{0}^{T} \hat{Z}(s)\cdot G(\hat{X}(s),\hat{X}(s+\theta/\Omega))\,ds.
\]
with $\hat{Z}$ and $\hat{X}$ as defined above.

To study cluster solutions we will make two simplifications.
First, we assume that all the delays are equal:
\begin{equation}
 \tau_{ij}=\tau,\ \mbox{i.e., } \eta_{ij}=\eta.
\label{simp1}
\end{equation}
Second, we will assume the network has some symmetry.  In particular,
we will consider
the coupling matrix to be in circulant form:
\begin{equation}
W = circ(w_0, w_1, w_2, \cdots, w_{N-1} ),\quad \mbox{equivalently,} \quad
W_{ij}=w_{j-i \modN}.
\label{simp2}
\end{equation}

Following \cite{MRTWBC}, we will say the network has connectivity radius $r$, if $w_k > 0$ for all $k\leq r$, and $w_k=0$ for all $k> r$. For example, a network with nearest neighbor coupling has connectivity radius $r=1$.  Our results will be derived with the coupling matrix 
\eqref{simp2}, but can be applied to coupling with any connectivity radius by 
setting the appropriate $w_k=0$.

Finally, we will assume there is no self coupling, thus $w_0=0$.
These simplifications will apply for the next two sections. In
section~\ref{sec:pert}, we will return to the general
model \eqref{DDE phase model 0}.


\section{Existence and stability of cluster solutions}
\label{section phase difference analysis}

Rewriting \eqref{DDE phase model 0} using the simplifications \eqref{simp1}-\eqref{simp2}
and dropping the higher order terms in $\epsilon$ we have
\begin{equation}\label{DDE phase model 1}
\frac{d\theta_i}{dt}=\Omega+\epsilon \sum_{j=1,j\ne i}^N w_{j-i \modN} H(\theta_j-\theta_i-\eta), \ i=1, 2, \cdots, N.
\end{equation}

Now the right hand sides of equation \eqref{DDE phase model 1}
depend only on the difference of phases. Thus, introducing the phase
difference variables:
\begin{eqnarray}\label{circulant PD variables}
\phi_i = \theta_{i+1}-\theta_i, \ i=1, \dots, N,
\end{eqnarray}
we can transform the phase equation (\ref{DDE phase model 1}), to the following system
\begin{eqnarray}\label{circulant PD DE}
\frac{d\phi_i}{dt} = \epsilon \sum_{k=1}^{N-1} w_k \bigg( H(\sum_{s=0}^{k-1}\phi_{i+s+1 \modN}-\eta)-H(\sum_{s=0}^{k-1}\phi_{i+s \modN}-\eta) \bigg)
\end{eqnarray}
for $i=1, 2, \cdots, N$.

Note that the $N$ phase difference variables are not independent but satisfy the relation
\begin{eqnarray}\label{constraint}
\sum_{i=1}^N \phi_i = 0 \ \mod 2\pi.
\end{eqnarray}
Thus, the $N-$dimensional system (\ref{circulant PD DE}) could be reduced to system
of dimension $N-1$. However, to take advantage of the symmetry, we choose instead to 
work with the
full set of $N$ equations and apply the constraint \eqref{constraint}.

As discussed above, a cluster solution of the DDE model \eqref{PM general network DDE}
is one where all the oscillators have the same waveform, but they separate into
different groups or clusters. Oscillators within a cluster are synchronized, while
oscillators in different clusters are phase-locked with some fixed phase difference.
It follows that in a cluster solution the difference between the phases of any two
oscillators are fixed. Using \eqref{DDE phase model 1} we can show that, to order $\epsilon$, these solutions correspond to the lines
\begin{equation}
 \theta_i=(\Omega+\epsilon \omega) t + \theta_{i0}. 
\label{phase solution}
\end{equation}
See \cite{Okuda93} for details of this calculation in the case that
$\eta=0$ and $w_k=w$. The case we are considering is completely analogous.
Further, from the definition (\ref{circulant PD variables}), it is clear that cluster
solutions correspond to equilibrium points of the phase difference equation (\ref{circulant PD DE}).  Therefore, by studying the existence of the equilibrium points of the phase difference model (\ref{circulant PD DE}), we can obtain the existence of the corresponding cluster solutions of the original DDE model.

For the sake of simplicity and generality, we focus our analysis on equilibrium solutions which are independent of $H$, and the $w_k$. It is clear from eq. (\ref{circulant PD DE}) that one such equilibrium point is given by $\phi_i=\psi$, $i=1, \dots, N$. Observe that the constraint
condition (\ref{constraint}) forces
\begin{eqnarray}
N\psi = 0 \ \text{mod $2\pi$}.
\end{eqnarray}
Different values of $\psi$ correspond to different cluster solutions. For example, $\psi=0$ corresponds to the in-phase or fully synchronized solution. When $N$ is even, $\psi=\pi$ corresponds to the anti-phase solution which is the state where oscillators segregate into two clusters and the two clusters oscillate with half-period phase difference. For a solution with more than two clusters, the value of $\psi$ determines the {\em ordering} of the
clusters/neurons in the solution.  Different values of $\psi$ can have the same number of clusters with different oscillators in the clusters and/or a different ordering
of the clusters in the solution. We shall see some examples of this in section~\ref{section application to N ML}.

\begin{theorem}[Existence of phase-locked solutions]\label{theorem psi and n}
The phase difference model (\ref{circulant PD DE}) admits $N$ equilibrium points of the form $\phi_i=\psi$, $i=1, \dots, N$:
\begin{itemize}

\item[(i)] $\psi=0$ corresponds to the 1-cluster, or fully synchronized solution.

\item[(ii)] $\psi =\frac{2p\pi}{N}$ where $p, N$ are relatively prime corresponds to an $N$-cluster, or splay solution.

\item[(iii)] $\psi = \frac{2m\pi}{n}$ where $1<n<N$ divides $N$ evenly, $1\le m<n$, and $m, n$ are relatively prime corresponds to a symmetric $n$-cluster solution.
\end{itemize}
If $\psi$ is a solution then so is $2\pi-\psi$ and they have the same number of clusters. The ordering of the clusters
of the $2\pi-\psi$ solution is the reverse of the $\psi$ solution.
\end{theorem}
\begin{proof}
The proof is similar to that found in \cite{MRTWBC}, hence we omit it.
\end{proof}

\begin{remark}
For any $N>2$, the in-phase and at least two splay solutions always exist.
For any even number $N$, the 2-cluster solution always exists.
\end{remark}

\subsection{Stability - general circulant coupling}
To study the stability, we linearize (\ref{circulant PD DE}) about the equilibrium point $\phi_i=\psi$, and obtain
\begin{eqnarray}\label{circulant PD linearization}
\frac{d\phi}{du} = J \phi,
\end{eqnarray}
where $\phi = (\phi_1, \dots, \phi_N)^T$, and the Jacobian matrix is a circulant matrix $J = circ(c_0, c_1, \dots, c_{N-1})$
\begin{displaymath}
c_k =
\left\{
\begin{array}
    {l@{\quad \quad}l}
w_kH^\prime(k\psi-\eta), \  1\leq k \leq N-1,\\
-\sum_{s=1}^{N-1}w_sH^\prime(s\psi-\eta), \ k=0.
\end{array}\right.
\end{displaymath}
A standard result for circulant matrices \cite{gray2006} shows that
the eigenvalues of $J$ are given by
\begin{eqnarray}
\lambda_j &=& c_0 + \sum_{k=1}^{N-1} c_k e^{\frac{2\pi i}{N}kj}\nonumber\\
&=&-\sum_{k=1}^{N-1}w_kH'(k\psi-\eta)(1- e^{\frac{2\pi i}{N}kj}),\ j=0,\ldots, N-1. \label{circulant PD J eigenvalues}
\end{eqnarray}
Note that there is always a zero eigenvalue ($\lambda_0$ = 0). For the
phase difference model this comes from the fact that the phase differences
are not independent. It can be verified that if the constraint \eqref{constraint} is
used to reduce the phase difference model \eqref{circulant PD DE} to $N-1$ equations
then the linearization yields only the eigenvalues $\lambda_j,\ j=1,\ldots,N-1$.
Thus stability of the equilibrium points is determined by these eigenvalues.

We note that linearizing \eqref{DDE phase model 1} about the corresponding solution
\eqref{phase solution} yields the same eigenvalues \eqref{circulant PD J eigenvalues}.
See \cite{Okuda93} for details of this calculation in the case that $\eta=0$ and
$w_k=w,k=0,1,\ldots,N-1$.  Recall that a cluster solution corresponds to a line in
the phase model. The zero eigenvalue corresponds to the motion along this line.

From the discussion above, the system has a synchronized solution corresponding to $\psi = 0$. The elements of the Jacobian matrix for this solution are
\begin{displaymath}
c_k =
\left\{
\begin{array}
    {l@{\quad \quad}l}
w_kH^\prime(-\eta), \  1\leq k \leq N-1,\\
-H^\prime(-\eta)\sum_{s=1}^{N-1}w_s, \ k=0.
\end{array}\right.
\end{displaymath}
It follows that the real parts of eigenvalues in (\ref{circulant PD J eigenvalues}) are
\begin{eqnarray}\label{circulant PD eigenvalues for 1C}
Re(\lambda_j) = -H^\prime(-\eta) \sum_{k=1}^{N-1}w_k (1-\cos \frac{2\pi kj}{N}).
\end{eqnarray}
This leads to the following result.

\begin{theorem}[Stability of the synchronized solution]\label{Antiphasetheorem}
The stability of the synchronized solution of the phase difference model
\eqref{circulant PD DE} is independent of the size of the network and coupling between oscillators ($w_k$). In particular, the synchronized solution is asymptotically stable when $H^\prime(-\eta) > 0$, and unstable when $H^\prime(-\eta) < 0$.
\end{theorem}



We know that when $N$ is even, the phase model always admits a 2-cluster solution, which corresponds to $\psi = \pi$ in the phase difference model. In this case, the Jacobian matrix satisfies
\begin{displaymath}
c_k =
\left\{
\begin{array}
    {l@{\quad \quad}l}
w_kH^\prime(\pi-\eta), \  k=1, 3, 5, \cdots, N-1,\\
w_kH^\prime(-\eta), \  k=2, 4, 6, \cdots, N-2,\\
-\sum_{s=1}^{N-1}c_s, \ k=0.
\end{array}\right.
\end{displaymath}
Therefore, the real parts of nonzero eigenvalues in (\ref{circulant PD J eigenvalues}) are given by
\[
Re(\lambda_{\frac{N}{2}}) = -2 H^\prime(\pi-\eta) \sum_{k=1, k \ odd}^{N-1}w_k
\]
and, for $j=1,\ldots, \frac{N}{2}-1,\frac{N}{2}+1,\ldots, N-1:$
\begin{eqnarray*}\label{circulant PD eigenvalues for 2C}
Re(\lambda_j) = -H^\prime(\pi-\eta) \sum_{k=1, k \ odd}^{N-1}w_k (1-\cos \frac{2\pi kj}{N}) -H^\prime(-\eta) \sum_{k=2, k \ even}^{N-2}w_k (1-\cos \frac{2\pi kj}{N}).
\end{eqnarray*}
This leads to the following
\begin{theorem}[Stability of the anti-phase solution]\label{circulant corollary 1C}
If $N$ is even the phase difference model \eqref{circulant PD DE}
admits the anti-phase cluster solution where adjacent
oscillators are out of phase by one half the period.  If $H'(\eta)>0$ and $H'(\pi-\eta)>0$ then this solution is asymptotically stable. If $H'(\pi-\eta)<0$ then this solution is unstable.
\end{theorem}

\begin{remark}
In the above stability results, we assume $\epsilon > 0 $. If $\epsilon < 0 $, the stability of asymptotically stable solutions and totally unstable solutions will be reversed, and the saddle type solutions will remain of saddle type.
\end{remark}


\subsection{Stability analysis for bi-directional, distance dependent coupling}

In this section, we consider a special case where the coupling strength is distance-dependent and bi-directional. In real neural networks, coupling strength is not necessarily determined by the physical distance. However, the ``distance" here can be generalized to include functional distance \cite{Li03}: the degree of correlation in the activity of coupled neurons. Therefore, we consider a
coupling matrix that satisfies
\begin{equation}
W=circ(0,w_1,w_2,\ldots,w_{N/2},\ldots,w_2,w_1)
\label{bidireven}
\end{equation}
if $N$ is even, and
\begin{equation}
W=circ(0,w_1,w_2,\ldots,w_{(N-1)/2},w_{(N-1)/2},\ldots,w_2,w_1)
\label{bidirodd}
\end{equation}
if $N$ is odd.

Applying the results above to this system we find that the elements
of the Jacobian matrix are
\begin{displaymath}
c_k =
\left\{
\begin{array}
    {l@{\quad \quad}l}
w_kH^\prime(k\psi-\eta), \  1\leq k \leq N/2,\\
w_{N-k} H^\prime(k\psi-\eta), \ k> N/2 \\
-\sum_{k=1}^{\frac{N}{2}-1}w_k\big(H^\prime(k\psi-\eta)+H^\prime((N-k)\psi-\eta)\big)-w_{\frac{N}{2}}H^\prime(\frac{N\psi}{2}-\eta),\ k=0
\end{array}\right.
\end{displaymath}
when $N$ is even, and
\begin{displaymath}
 c_k =
\left\{
\begin{array}
    {l@{\quad \quad}l}
w_kH^\prime(k\psi-\eta), \  1\leq k \leq (N-1)/2,\\
w_{N-k} H^\prime(k\psi-\eta), \ k> (N-1)/2 \\
-\sum_{k=1}^{(N-1)/2}w_k\big(H^\prime(k\psi-\eta)+H^\prime((N-k)\psi-\eta)\big),\ k=0
\end{array}\right.
\end{displaymath}
when $N$ is odd.

Recall that $\psi$ and $2\pi-\psi$ correspond to the same type of cluster
solution.  For a network with bi-directional coupling, these solutions have a
stronger relationship.
\begin{theorem}
For the phase model with coupling matrix given by \eqref{bidireven} or \eqref{bidirodd}, the solutions $\phi_i =\psi$ and $\phi_i = 2\pi-\psi$ have the same stability.
\end{theorem}

\begin{proof}
Denote the Jacobian matrix for the linearization equation at $\phi_i = \psi$ and $\phi_i = 2\pi-\psi$ to be $J=circ( c_0, c_1, \dots, c_{N-1})$, and $\tilde J = circ(\tilde c_0, \tilde c_1, \dots, \tilde c_{N-1} )$, respectively. By Theorem \ref{theorem psi and n}, we know that there are $N$ possible $\psi$ values, $\psi = \frac{2k\pi}{N}$, $k = 0, 1, \dots, N-1$. Therefore,
\begin{eqnarray*}
\tilde c_{N-1} = w_{N-1}H^\prime (\frac{(N-k)(N-1)2\pi}{N}-\eta) = w_{N-1}H^\prime(\frac{2\pi m}{N}-\eta).
\end{eqnarray*}
Since $w_{N-1} = w_1$, we have $\tilde c_{N-1} = c_1$. Similarly, we have $\tilde c_0 = c_0$, and $\tilde c_{N-2} = c_2$, \dots. Therefore, $J$ and $\tilde J$ have the same eigenvalues. Thus, $\psi$ and $2\pi -\psi$ have the same stability.

\end{proof}

A special case of bi-directional coupling is when the only nonzero coupling
coefficient is $w_1$. This is commonly called nearest-neighbour coupling.
In this case the stability of any symmetric cluster solution is easily determined.
\begin{theorem}
For the phase model with coupling matrix given by \eqref{bidireven} or \eqref{bidirodd}
with $w_1\ne0$ and $w_j=0,\ j=2,\ldots,N$, the the symmetric cluster
solution with $\phi_i =\psi$ is asymptotically stable if
$H'(\psi-\eta)+H'(-\psi-\eta)>0$ and unstable if $H'(\psi-\eta)+H'(-\psi-\eta)<0.$
\end{theorem}
\begin{proof}
A straightforward calculation from \eqref{circulant PD J eigenvalues}
shows that the real parts of the eigenvalues of the solution  $\phi_i =\psi$
are given by
\[ {\rm Re}(\lambda_j)=-w_1\left[H'(\psi-\eta)+H'(-\psi-\eta)\right](1-\cos(2\pi j/N)),\ j=1,\ldots, N-1 \]
The result follows.
\end{proof}
Note that this is an extension of a result of \cite{MRTWBC} to the case when the
coupling is delayed.


\subsection{Stability analysis for global homogeneous coupling}

We next consider a special case: $W_1= circ(0, 1, \cdots, 1)$. That is, all the coupling weights are the same.
A straightforward calculation show that the eigenvalues \eqref{circulant PD J eigenvalues}
for a symmetric $n$-cluster solution in this case can be written as follows:
\begin{eqnarray}
\begin{aligned}
\lambda_0& = 0,\\
{\lambda}^{(n)}_0 &= -\frac{N}{n}\sum_{k=0}^{n-1}H^\prime(\frac{2\pi k}{n} - \eta),\  \text{multiplicity $N-n$},\\
{\lambda}^{(n)}_j&=-\frac{N}{n}\sum_{k=0}^{n-1}H^\prime(\frac{2\pi k}{n} - \eta)(1-e^{i2\pi kj/n}), \ p=1, \cdots, n-1.
\end{aligned}
\label{homevals}
\end{eqnarray}
This is identical to what was shown in \cite{WS15}, where they made the following
observation. The stability of an
$n$-cluster solution (with $n<N$) depends on the number of clusters and
the phase differences, not the size of the network. For example,
any network with $N=3m$ ($m$ a positive integer) has a $3$-cluster
solution with $\psi=2\pi/3$. The stability of this solution is the
same for all networks with $m>1$.

\begin{remark}
As discussed in \cite{WS15}, since networks with global homogeneous
coupling are unchanged by
any rearrangement of the indices, there are many more cluster solutions.
For example, consider a network where $N>2$ is even.
When the connection matrix is circulant with different $w_k$,
there is one $2$-cluster solution with
oscillators $1,3,5,\ldots, N-1$ forming one cluster and
oscillators $2,4,\ldots N$ forming the second cluster. For a network
with global homogeneous coupling, {\em any} division of the oscillators
into two groups of $N/2$ oscillators is an admissible $2$-cluster
solution with stability described by  \eqref{homevals} with $n=2$.
\end{remark}


\subsection{Other types of cluster solutions}

If more conditions are put on the coupling matrix then different cluster
solutions may occur. For example, consider a 2-cluster solution where the
phase differences between adjacent elements is not the same, but is
described by
\begin{equation}
\phi_1 = \phi_3 = \cdots = \phi_{N-1} = 0, \ \text{and } \phi_2 = \phi_4 = \cdots = \phi_N = \pi, \label{AC2 solution 1}
\end{equation}
or
\begin{equation}
\phi_1 = \phi_3 = \cdots = \phi_{N-1} = \pi, \ \text{and } \phi_2 = \phi_4 = \cdots = \phi_N = 0. \label{AC2 solution 2}
\end{equation}
In this situation the elements group into pairs, so that each element is synchronized
with one of its nearest neighbours and one-half period out of phase with its other
nearest neighbour.
As shown by the next result, these solutions exist under
appropriate conditions on the connectivity matrix.
\begin{theorem}\label{theorem AC2 existence condt}
For a network with a circulant connectivity matrix, the system (\ref{circulant PD DE}) admits solutions of the form (\ref{AC2 solution 1}) and (\ref{AC2 solution 2}) if $N=4p$ for some integer $p$, and $\sum_{k=0}^{p-1} w_{4k+1} = \sum_{k=0}^{p-1} w_{4k+3} $.
\end{theorem}

\begin{proof}
Applying the constraint condition (\ref{constraint}) to \eqref{AC2 solution 1}
or \eqref{AC2 solution 2}, we have that, for some integer $p$,
\begin{eqnarray*}
\frac{N}{2}\cdot \pi = 2p \pi.
\end{eqnarray*}
Therefore, $N=4p$, for some integer $p$.

Substituting solution (\ref{AC2 solution 1}) or (\ref{AC2 solution 2}) into the system (\ref{circulant PD DE}), we have that
\begin{eqnarray*}
\sum_{k=0}^{p-1} w_{4k+1} \big( H(\pi-\eta) - H(-\eta) \big) = \sum_{k=0}^{p-1} w_{4k+3} \big( H(\pi-\eta) - H(-\eta) \big).
\end{eqnarray*}
To satisfy this for any $H$, we must have $\sum_{k=0}^{p-1} w_{4k+1} = \sum_{k=0}^{p-1} w_{4k+3} $.

\end{proof}

\begin{remark}
Note that, for networks with bi-directional coupling or global homogeneous coupling, the second condition, $\sum_{k=0}^{p-1} w_{4k+1} = \sum_{k=0}^{p-1} w_{4k+3} $, is automatically satisfied if $N=4p$.
\end{remark}

We consider the $2$-cluster solutions in the form of (\ref{AC2 solution 1}) first. Assume the conditions of Theorem \ref{theorem AC2 existence condt} are satisfied. Linearizing the system (\ref{circulant PD DE}) at (\ref{AC2 solution 1}), we obtain that
\begin{eqnarray}\label{AC2 linear DE L}
\frac{d\phi}{dt} = \epsilon L \phi,
\end{eqnarray}
where the Jacobian matrix has the form

\parbox{10cm}{
\begin{displaymath}
L =
\begin{pmatrix}
\alpha_0 & \alpha_1 & \alpha_2 & \alpha_3 & \cdots & \alpha_{N-1}\\
\beta_{N-1} & \beta_0 & \beta_1 & \beta_2 & \cdots & \beta_{N-2} \\
\alpha_{N-2} &\alpha_{N-1} & \alpha_0 & \alpha_1 & \cdots & \alpha_{N-3}\\
\beta_{N-3} & \beta_{N-2} & \beta_{N-1}  & \beta_0 & \cdots & \beta_{N-4} \\
\vdots & \vdots & \vdots & \vdots & \ddots & \vdots \\
\alpha_2 & \alpha_3 & \alpha_4 & \alpha_5 & \cdots & \alpha_1 \\
\beta_1 & \beta_2 &\beta_3 & \beta_4 & \cdots & \beta_0
\end{pmatrix},
\end{displaymath}
}\hfill
\parbox{1cm}{\begin{eqnarray}\label{AC2 Jacobian 1}\end{eqnarray}}

\noindent with
\begin{eqnarray*}
\begin{aligned}
\alpha_0 &= -H^\prime (-\eta)\big(\sum_{k=0}^{p-1} w_{4k+1}+ \sum_{k=1}^{p-1} w_{4k}\big)-H^\prime (\pi-\eta) \sum_{k=0}^{p-1} (w_{4k+2}+w_{4k+3}),\\
\beta_0 &= -H^\prime (-\eta)\big(\sum_{k=0}^{p-1} w_{4k+3}+ \sum_{k=1}^{p-1} w_{4k}\big)-H^\prime (\pi-\eta) \sum_{k=0}^{p-1} (w_{4k+1}+w_{4k+2}),
\end{aligned}
\end{eqnarray*}
and for $k=1, \cdots, N-1$
\begin{displaymath}
\alpha_k =
\left\{
\begin{array}
    {l@{\quad \quad}l}
w_k H^\prime (\pi-\eta) + \big(H^\prime (\pi-\eta) - H^\prime (-\eta) \big)\big( \sum_{j=s+1}^{p-1} w_{4j+1} - \sum_{j=s}^{p-1} w_{4j+3}  \big), \  k=4s+1, 4s+2 \\
w_k H^\prime (-\eta) + \big(H^\prime (\pi-\eta) - H^\prime (-\eta) \big)\big( \sum_{j=s+1}^{p-1} w_{4j+1} - \sum_{j=s+1}^{p-1} w_{4j+3}  \big), \  k=4s+3, 4s\\
\end{array}\right.
\end{displaymath}

for all the possible $s$ values, and
\begin{displaymath}
\beta_k =
\left\{
\begin{array}
    {l@{\quad \quad}l}
w_k H^\prime (-\eta) - \big(H^\prime (\pi-\eta) - H^\prime (-\eta) \big)\big( \sum_{j=s+1}^{p-1} w_{4j+1} - \sum_{j=s}^{p-1} w_{4j+3}  \big), \  k=4s+1, \\
w_k H^\prime (\pi-\eta) - \big(H^\prime (\pi-\eta) - H^\prime (-\eta) \big)\big( \sum_{j=s+1}^{p-1} w_{4j+1} - \sum_{j=s}^{p-1} w_{4j+3}  \big), \  k=4s+2, \\
w_k H^\prime (\pi-\eta) -\big(H^\prime (\pi-\eta) - H^\prime (-\eta) \big)\big( \sum_{j=s+1}^{p-1} w_{4j+1} - \sum_{j=s+1}^{p-1} w_{4j+3}  \big), \  k=4s+3\\
w_k H^\prime (-\eta) -\big(H^\prime (\pi-\eta) - H^\prime (-\eta) \big)\big( \sum_{j=s+1}^{p-1} w_{4j+1} - \sum_{j=s+1}^{p-1} w_{4j+3}  \big), \  k=4s\\
\end{array}\right.
\end{displaymath}
for all the possible $s$ values.

\begin{remark}
Solutions (\ref{AC2 solution 1}) and (\ref{AC2 solution 2}) have the same stability. The Jacobian matrix of the linearization of system (\ref{circulant PD DE}) at (\ref{AC2 solution 2}) is in the form

\parbox{10cm}{
\begin{displaymath}
\hat L =
\begin{pmatrix}
\beta_0 & \beta_1 & \beta_2 & \beta_3 & \cdots & \beta_{N-1}\\
\alpha_{N-1} & \alpha_0 & \alpha_1 & \alpha_2 & \cdots & \alpha_{N-2} \\
\beta_{N-2} &\beta_{N-1} & \beta_0 & \beta_1 & \cdots & \beta_{N-3}\\
\alpha_{N-3} & \alpha_{N-2} & \alpha_{N-1}  & \alpha_0 & \cdots & \alpha_{N-4} \\
\vdots & \vdots & \vdots & \vdots & \ddots & \vdots \\
\beta_2 & \beta_3 & \beta_4 & \beta_5 & \cdots & \beta_1 \\
\alpha_1 & \alpha_2 &\alpha_3 & \alpha_4 & \cdots & \alpha_0
\end{pmatrix},
\end{displaymath}
}\hfill
\parbox{1cm}{\begin{eqnarray}\label{AC2 Jacobian 2}\end{eqnarray}}

\noindent which is equivalent to $L$.
\end{remark}

We were not able to obtain general results about the eigenvalues of 
(\ref{AC2 Jacobian 1}) and (\ref{AC2 Jacobian 2}).  Thus, we are not able to make any 
general conclusions about the stability of solutions (\ref{AC2 solution 1}) 
and (\ref{AC2 solution 2}). However, in particular cases the eigenvalues can be calculated 
numerically from the expressions above. We will do this for the example in the next section.


\section{Application to networks of Morris-Lecar oscillators with global synaptic coupling}
\label{section application to N ML}
In this section, we apply our results to a specific network: globally coupled Morris-Lecar oscillators. Since the nondimensional form of Morris-Lecar equation is more convenient to work with, we adopt the dimensionless Morris-Lecar model which is formulated by Rinzel and Ermentrout in \cite{RE89} and used in Campbell and Kobelevskiy \cite{CK12}. Considering $N$ identical Morris-Lecar oscillators with delayed synaptic coupling, we have the following model
\begin{eqnarray}
v_i^\prime &=& I_{app}-g_{Ca} m_\infty(v_i)(v_i-v_{Ca})-g_K w_i(v_i-v_K)
\label{six Morris-Lecar model}
\\
&&\hspace{-.2in}-g_L(v_i-v_L)-\frac{g_{syn}}{N-1}\sum_{j=1, j\neq i}^{N} s(v_j(t-\tau))(v_i(t)-E_{syn}), \nonumber \\
w_i^\prime &=& \varphi \lambda(v_i)(w_\infty(v_i)-w_i),\nonumber \
\end{eqnarray}
where $i=1, \dots, N$ and
\begin{eqnarray*}
m_\infty(v)&=&\frac{1}{2} (1+\tanh((v-\nu_1)/\nu_2)),\\
w_\infty(v)&=&\frac{1}{2} (1+\tanh((v-\nu_3)/\nu_4)),\\
\lambda(v)&=&\cosh((v-\nu_3)/2\nu_4),\\
s(v)&=&\frac{1}{2}(1+\tanh(10v)).
\end{eqnarray*}
Using the parameter set I from \cite[Table 1]{CK12}, when there is no coupling in the network each oscillator has a unique exponentially asymptotically stable limit cycle with period $T\approx 23.87$ corresponding to $\Omega=0.2632$.

\begin{table}[ht]
\centering
\begin{tabular}{|c c c|}
\hline
Parameter & Name  & value \\ [0.5ex]
\hline
$v_{Ca}$      & Calcium equilibrium potential & 1\\
$v_K$          & Potassium equilibrium potential     & -0.7   \\
$v_L$     & Leak equilibrium potential      & -0.5   \\
$g_K$          & Potassium ionic conductance & 2 \\
$g_L$          & Leak ionic conductance   & 0.5 \\
$\varphi$      & Potassium rate constant  & $\frac{1}{3}$ \\
$\nu_1$          & Calcium activation potential    & -0.01   \\
$\nu_2$       & Calcium reciprocal slope & 0.15 \\
$\nu_3$     & Potassium activation potential     & 0.1   \\
$\nu_4$          & Potassium reciprocal slope & 0.145 \\
$g_{Ca}$            & Calcium potential conductance     & 1 \\
$I_{app}$       & Applied current      & 0.09 \\
[1ex]
\hline
\end{tabular}
\caption{Parameters used in system (\ref{six Morris-Lecar model}) \cite[Table 1]{CK12}}
\label{table parameters}
\end{table}

\subsection{Phase model analysis}

The calculation of the phase model interaction function, $H$, described
in section 2, may be carried out numerically.  We used the numerical
simulation package XPPAUT \cite{Ermentrout02} to do this
for model (\ref{six Morris-Lecar model}) with $\tau=0$, and to calculate a finite number of terms in the Fourier
series approximation for $H$. This gives an explicit approximation for
$H$:
\begin{equation}
H(\phi)\approx a_0+\sum_{k=1}^K (a_k\cos(k\phi)+b_k\sin(k\phi)).
\label{Happ}
\end{equation}
The first nine terms of Fourier coefficients are shown in Table \ref{table HFourier coeff}.
Figure \ref{figure Hplot} shows the plot of the interaction function (red solid), $H$, together with the approximations using one (black solid) and 20 terms (green dashed) of Fourier Series. Obviously, the one term approximation is not enough to explain the behavior of $H$. However, the 20-term approximation is indistinguishable with the numerically calculated $H$. Therefore, we adopt the 20-term approximation for subsequent calculations.
\begin{table}[ht]
\centering
\begin{tabular}{| c | c | c || c | c | c |}
\hline
$k$ & $a_k$  & $b_k$  & $k$ & $a_k$  & $b_k$ \\ [0.5ex]
\hline
0 & -2.0214064 & 0 &5 & -0.01054942& 0.010251001\\
1 & 1.994447 &-0.93897837 &6 & -0.002131111& 0.0046384884\\
2 & 0.010604496& 0.27575842& 7&9.9814584e-05 & 0.0013808256\\
3 &-0.051657807 & 0.042355601&8 &0.00015646126 &7.391713e-05 \\
4 & -0.029127343&0.01801952 & 9& -8.1846403e-05& -0.00024995379\\
[1ex]
\hline
\end{tabular}
\caption{Fourier coefficients of the interaction function for model (\ref{six Morris-Lecar model}).}
\label{table HFourier coeff}
\end{table}

\begin{figure}[!htbp]
\centering
\includegraphics[scale=0.3]{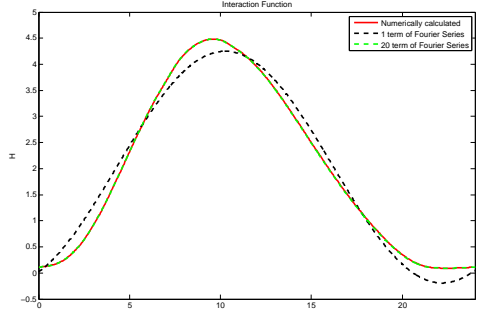}
\caption{Interaction function for model (\ref{six Morris-Lecar model}) and the approximations using 1 and 20 terms of Fourier Series}\label{figure Hplot}
\end{figure}

With the explicit approximation for $H$ \eqref{Happ} and the value of
the coefficients $a_j,b_j$, we can determine the asymptotic stability of any possible symmetric cluster states for any $N$ using the eigenvalues (\ref{circulant PD J eigenvalues}) calculated in the last section.
In this section, we consider two coupling matrices
\begin{eqnarray}\label{ML define W1 W2}
W_1 &=& circ(0, 1, \frac{1}{2}, \frac{1}{3}, \cdots, \frac{1}{2}, 1), \ \text{bi-directional, distance dependent }\\
W_2 &=& circ(0, 1, 1, \cdots, 1), \ \text{global homogeneous}.
\end{eqnarray}

With the coupling matrices $W_1$ and $W_2$, various values of $\epsilon$ and the 
time delay $\tau$, we used our phase model results above to predict the stability of all 
possible symmetric cluster solutions for $N=2, \cdots, 10$. The results
are shown in Tables \ref{table W1} \ref{table W2}. 
\begin{sidewaystable}
\centering
\resizebox{1\textwidth}{!}{
\begin{tabular}{|c|c|c|c|c|c|}
\hline
\multirow{2}{*}{N} & \multirow{2}{*}{$n$} & \multirow{2}{*}{$\psi$} &\multirow{2}{*}{Phase model prediction}
& \multicolumn{2}{|c|}{Full model} \\
\cline{5-6}
&&&&  $\epsilon = 0.01$ & $\epsilon = 0.05$ \\
\hline
\multirow{3}{*}{4} & 1 & 0 & $(0, 1.53)\cup(14.28, 23.87)$
& --
& -- \\

& 2 & $\pi$ & $(2.47, 10.46)$
& $(2.20, 10.21)$
& $(1.68, 9.32)\cup (17.47, 23.87)$
 \\

& 4 & $\frac{\pi}{2}, \frac{3\pi}{2}$ &(0.57,3.22)$\cup$(8.69,14.69) 
&  (0,2.96)$\cup$(8.36,14.16)& (0,2.26)$\cup$(6.86,12.36) \\
\hline

\multirow{3}{*}{5} & 1 & 0 & $(0, 1.53)\cup(14.28, 23.87)$  & -- & --  \\
& 5 & $\frac{2\pi}{5},\frac{8\pi}{5}$ &(1.26,2.48)$\cup$(10.84,13.46)
 & (0,2.21)$\cup$(10.31,12.71) & (0,1.51)$\cup$(8.71,10.81) \\

& 5 & $\frac{4\pi}{5},\frac{6\pi}{5}$ & (1.66,3.66)$\cup$(4.26,13.09)& (1.513,12.61) & (0.70,11.49) \\

\hline

\multirow{3}{*}{6} & 1 & 0 & $(0, 1.53)\cup(14.28, 23.87)$ 
& $(0, 1.41)\cup (12.31, 23.87)$
& $(0, 1.70)\cup (7.82, 23.87)$
\\

& 2 & $\pi$ 
& $(2.64, 9.45)$ 
& $(2.30, 9.10)$
& $(1.58, 7.79)\cup (16.59, 27.31)$
\\

& 3 & $\frac{2\pi}{3}$, $\frac{4\pi}{3}$ 
& $(0.41, 13.31)$
& $(0.41, 12.91)$
& $(0, 4.19) \cup (5.30, 11.40) \cup (17.41, 20.9) \cup (22.31, 23.87)$

\\

& 6 & $\frac{\pi}{3},\frac{5\pi}{3}$&(0.58,0.87)$\cup$(12.32,14.10) 
& (0,1.51)$\cup$(12.01,13.11) & (0,1.11)$\cup$(9.21,10.31) \\
\hline

\multirow{4}{*}{7} & 1 & 0 & $(0, 1.53)\cup(14.28, 23.87)$ 
& $(0, 1.49) \cup (12.19, 23.87)$
& $(0, 1.68) \cup (7.52, 23.87)$ \\

& 7 & $\frac{2\pi}{7},\frac{12\pi}{7}$ & (12.82,13.86)  & (0,1.21)$\cup$(12.11,12.81) & (0,1.10)$\cup$(8.82,9.82) \\

& 7 & $\frac{4\pi}{7},\frac{10\pi}{7}$&(2.33,4.37)$\cup$(7.59,13.83) & (0.51,3.91)$\cup$(7.21,13.11) & (0,2.71)$\cup$(5.81,11.11)\\
& 7 &$\frac{6\pi}{7},\frac{8\pi}{7}$&(2.51,3.45)$\cup$(4.04,4.93)$\cup$(5.48,5.96)$\cup$(7.47,13.13) & (2.51,4.91)$\cup$(6.91,12.11) & (1.70,3.81)$\cup$(5.70,10.82) \\

\hline
\multirow{5}{*}{8} & 1 & 0 & $(0, 1.53)\cup(14.28, 23.87)$ 
& $(0, 1.44) \cup (12.04, 23.87)$
& $(0, 1.74) \cup (7.27, 23.87)$ \\

& 2 & $\pi$ & $(2.63, 9.53)$ 
& $(2.25, 9.05)$
& $(1.55, 7.45) \cup (15.73, 23.87)$
  \\

& 4 & $\frac{\pi}{2}, \frac{3\pi}{2}$ & $(1.71, 3.22)\cup(8.69, 14.57)$ 
& $ (0.31, 2.81)\cup(8.11, 13.71)$
& $ (0, 1.80)\cup(6.21, 11.20)$
 \\

& 8 & $\frac{\pi}{4},\frac{7\pi}{4}$ & (13.34,13.95) & (0,1.01)$\cup$(12.31,12.71) & (0,1.00)$\cup$(8.52,9.42) \\
& 8 & $\frac{3\pi}{4},\frac{5\pi}{4}$ & (3.96,13.13)& (3.41,12.41) & (0.11,0.71)$\cup$(2.61,10.82)\\

\hline

\multirow{5}{*}{9} & 1 & 0 & $(0, 1.53)\cup(14.28, 23.87)$
& $(0, 1.66) \cup (11.93, 23.87)$
& $(0, 1.73) \cup (7.06, 23.87)$
\\

& 3 & $\frac{2\pi}{3}$, $\frac{4\pi}{3}$ & $(0.41, 5.04)\cup(8.08, 12.93)$ 
& $(0.41, 4.61)\cup (7.71, 12.41)$
& $(0, 3.30)\cup (5.80, 10.60)\cup (16.61, 19.31)$ \\

& 9 & $\frac{2\pi}{9},\frac{16\pi}{9}$ & (13.46,14.01) & --- & ---  \\
& 9 & $\frac{4\pi}{9},\frac{14\pi}{9}$ &(2.50,2.57)$\cup$(9.81,13.94) & (0.41,2.61)$\cup$(9.11,13.01) & (0,1.71)$\cup$(6.51,10.41)  \\
& 9 & $\frac{8\pi}{9},\frac{10\pi}{9}$ &(2.90,3.77)$\cup$(8.08,11.38)  &
(2.61, 4.01)$\cup$(7.51,11.01) & (1.60,3.12)$\cup$(5.92,8.81)\\
\hline
\end{tabular}
}
\caption{Comparison of phase model prediction of $\tau$-intervals of asymptotic stability for $n$-cluster solution with numerical of the full model. The coupling matrix is $W_1$. Other parameter values are given in Table \ref{table parameters}.}
\label{table W1}
\end{sidewaystable}

\begin{sidewaystable}\footnotesize
\centering
\resizebox{1\textwidth}{!}{
\begin{tabular}{|c|c|c|c|c|}
\hline
\multirow{2}{*}{N} & \multirow{2}{*}{n} &  \multirow{2}{*}{Phase model prediction}
& \multicolumn{2}{|c|}{Full model} \\
\cline{4-5}
&&& $\epsilon = 0.01$ & $\epsilon = 0.05$ \\
\hline
\multirow{2}{*}{2} & 1 & $(0, 1.53)\cup(14.28, 23.87)$ 
& $(0, 1.46)\cup (13.56, 23.87)$
& $(0, 1.43)\cup (11.53, 23.87)$
 \\

& 2 & (2.35,13.46)  & (2.23,13.43) & (1.92,13.32) \\

\hline

\multirow{2}{*}{3} & 1 & $(0, 1.53)\cup(14.28, 23.87)$ 
& $(0, 1.48)\cup(13.09, 23.87)$
& $(0, 1.52)\cup(9.53, 23.87)$
 \\

& 3 & (0.41,13.74)  & (0.50,13.40) & (0,12.6) \\

\hline
\multirow{3}{*}{4} & 1 & $(0, 1.53)\cup(14.28, 23.87)$ 
& $(0, 1.47)\cup (12.57, 23.87)$
& $(0, 1.70)\cup (8.11, 23.87)$
\\

&  2 & $(2.73, 9.19)$ 
& $(2.41, 8.91)$
& $(1.71, 7.71)\cup (17.53, 23.87)$
\\

& 4 & (1.93,3.22) $\cup$ (8.69,14.47)  & (0.97,2.87)$\cup$(8.47,13.97)& (0,1.96)$\cup$(6.97,12.27) \\

\hline

\multirow{2}{*}{5} & 1 & $(0, 1.53)\cup(14.28, 23.87)$ 
& $(0, 1.49) \cup (11.99, 23.87)$
& $(0, 1.79) \cup (7.22, 23.87)$
\\

& 5 & (1.57,2.69)$\cup$(9.76,13.20) & (0.93,2.23)$\cup$(9.13,12.43)
& (0,1.32)$\cup$(6.13,10.42)\\
\hline

\multirow{4}{*}{6} & 1 & $(0, 1.53)\cup(14.28, 23.87)$
& $ (0, 1.46) \cup (11.56, 23.87)$ 

&  \\

& 2 & $(2.73, 9.19)$ 
& $(2.30, 8.51)$
& $(1.48, 6.29) \cup (15.0, 23.19)$
\\

& 3 & $(0.41, 4.83)\cup (8.29, 12.79)$
&  $(0.28, 4.18) \cup (7.98, 11.98)$
&  $(0, 3.03) \cup (5.54, 9.83)\cup (15.94, 18.44)$  \\

& 6 & (12.26,13.86) & (11.96,12.72) & (0,0.91)$\cup$(9.21,9.91)\\

\hline
\multirow{2}{*}{7} & 1 & $(0, 1.53)\cup(14.28, 23.87)$
& $ (0, 1.49) \cup (11.01, 23.87)$
& $ (0, 1.94) \cup (6.10, 23.87)$ \\

& 7 & (12.47,13.54) & (11.92,12.32) & (0,0.92)$\cup$(8.52,9.32)\\
\hline

\multirow{4}{*}{8} & 1 & $(0, 1.53)\cup(14.28, 23.87)$
& $(0, 1.50)\cup (10.70, 23.87)$
& $(0, 2.00)\cup (5.70, 23.87)$ \\

& 2 & $(2.73, 9.19)$
& $ (2.22, 8.22) $
& $ (1.34, 5.44) \cup (13.24, 20.34)$ \\

& 4 & $(1.94, 3.22)\cup(8.69, 9.35)\cup(12.37, 14.47)$
& $ (0.53, 2.63) \cup (7.53, 8.43) \cup (11.13, 13.03) \cup (22.63, 23.23)$
& $ (0, 1.33) \cup (5.13, 6.03) \cup (7.33, 9.14) \cup (19.04, 20.34) \cup (20.94, 21.64)$ \\

& 8 & All unstable & (0,1.35)$\cup$(6.15,6.95)$\cup$(16.36,17.56) & (0,0.95)$\cup$(7.95,8.75)\\

\hline
\multirow{3}{*}{9} & 1 & $(0, 1.53)\cup(14.28, 23.87)$
& $(0, 1.48) \cup (10.22, 23.87)$
& $(0, 2.08) \cup (5.40, 23.87)$\\

& 3 & $(0.41, 4.83)\cup (8.29, 12.79)$
& $ (0.19,4.00) \cup (7.49, 11.29) \cup (21.39, 23.87)$
& $ (0, 2.78) \cup (5.00, 8.29) \cup (14.39, 15.80) \cup (17.50, 22.40)$\\

& 9 & (13.30,13.65) & (0.56,1.16)$\cup$(11.36,11.96) & (0,0.96)$\cup$(7.56,8.37)\\

\hline
\end{tabular}
}
\caption{Comparison of phase model prediction of $\tau$-intervals of asymptotic stability for $n$-cluster solution with numerical of the full model. The coupling matrix is $W_2$. Other parameter values are given in Table \ref{table parameters}.}
\label{table W2}
\end{sidewaystable}

\subsection{Numerical studies}

Numerical continuation studies of the full model (\ref{six Morris-Lecar model}) were carried out in DDE-BIFTOOL \cite{DDE-BIF} in MATLAB. This package allows one to compute branches of 
periodic orbits and their stability as a parameter is varied. Using the delay as a continuation
parameter, we used this package to compute the stability of all possible symmetric cluster 
solutions for $N=2, 3, \cdots, 10$ with the two different coupling matrices $W_1$, $W_2$ and 
four different values of $\epsilon$, $\epsilon = 0.001, 0.01, 0.05, 0.1$.  These results 
indicated that the phase model prediction is accurate up to $\epsilon = 0.01$.  
The results for $\epsilon = 0.01, 0.05$ are shown in Tables {\ref{table W1}} and 
\ref{table W2}. 

Using dde23 in MATLAB, we are able to numerically simulate the solution for larger sizes of networks. In the following, we show several numerical simulations that verify the predictions of the phase model for the case of a network with $N=140$ oscillators. This network admits 1-cluster, 2-cluster, 5-cluster, 7-cluster, 10-cluster, 14-cluster, 35-cluster, 70-cluster, and 140-cluster solutions.
From the phase model analysis, we are able to predict the stability regions for all the cluster states. Table \ref{table N140 phase model} summarize the stability intervals with respect to $\tau$ for the first five cluster types. 
\begin{table}
\centering

\begin{tabular}{|c|c|c|c|}
\hline
\multirow{2}{*}{n} & \multirow{2}{*}{$\psi$}
& \multicolumn{2}{|c|}{Stability w.r.t. $\tau$} \\
\cline{3-4}
&& $W_1$ & $W_2$ \\
\hline
1 & $0$ & (0, 1.52) $\cup$ (14.28,23.87) & (0, 1.52) $\cup$ (14.28,23.87)\\
\hline
2& $\pi$ & (2.73, 9.19) & (2.73, 9.19)\\
\hline
\multirow{2}{*}{5} & $\frac{2\pi}{5},\frac{8\pi}{5}$ & (1.52, 2.61) $\cup$ (10.78, 12.55) & \\
& $\frac{4\pi}{5},\frac{6\pi}{5}$ & (1.61, 2.81) $\cup$ (6.21, 7.77) $\cup$ (10.03, 12.55)& (1.57, 2.69) $\cup$ (10.03, 12.54) \\
\hline
\multirow{3}{*}{7} & $\frac{2\pi}{7},\frac{12\pi}{7}$ & (12.77, 13.29) & \\
&$\frac{4\pi}{7},\frac{10\pi}{7}$ & (8.13, 9.81 ) $\cup$ (11.12, 13.28) & (12.47, 13.28)\\
&$\frac{6\pi}{7},\frac{8\pi}{7}$ & (8.45, 9.88) $\cup$ (11.11, 13.13) &\\
\hline
\multirow{2}{*}{10}  & $\frac{\pi}{5},\frac{9\pi}{5}$ & All unstable & \\
& $\frac{3\pi}{5},\frac{7\pi}{5}$ & (7.85, 7.86) $\cup$ (11.80, 12.62) & All unstable\\
\hline
\end{tabular}

\caption{Phase model prediction of intervals of $\tau$ where stable 1-, 2-, 5-, 7-, and 10-cluster solutions exist. The network has 140 oscillators and the coupling matrix $W_1$ or $W_2$.}
\label{table N140 phase model}
\end{table}

The phase model predicts that, for bidirectional coupling, there should be four stable 5-cluster solutions when $\tau = 12$ corresponding to $\psi = \frac{k\pi}{5}$, $k=1,2, 3, 4$. 
In these 5-cluster solutions, the clusters are the same and given by
\begin{eqnarray*}
C_1= \{1, 6, 11, \dots, 136 \},\\
C_2= \{2, 7, 12, \dots, 137 \},\\
\vdots\\
C_5= \{5, 10, 15, \dots, 140 \}.
\end{eqnarray*}
but each solution has a different cluster ordering.  The ordering is
$C_1- C_2-C_3- C_4- C_5 $ with $\psi=2\pi/5$ (see Figure \ref{figure 5C orders} (a)),
$C_1- C_4- C_2-C_5- C_3$ with $\psi=4\pi/5$ (see Figure \ref{figure 5C orders} (b)),
$C_1- C_3- C_5- C_2- C_4 $ with $\psi=6\pi/5$(see Figure \ref{figure 5C orders} (c))
and
$C_1-C_5- C_4- C_3- C_2 $ with $\psi=8\pi/5$ (see Figure \ref{figure 5C orders} (d)).
Note that in Figure \ref{figure 5C orders} we reorder the indices so that 
oscillators that belong to the same cluster are plotted together.
\begin{figure}[!htbp]
\subfigure[$\psi=2\pi/5$]{\includegraphics[width=2.5in]{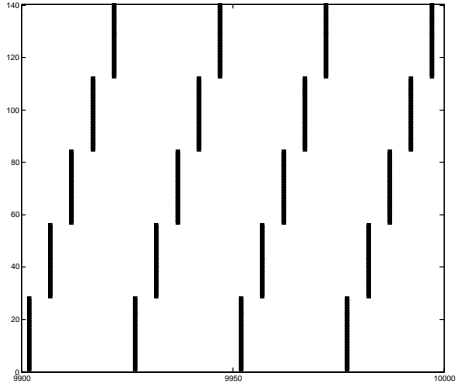}}
\subfigure[$\psi=4\pi/5$]{\includegraphics[width=2.5in]{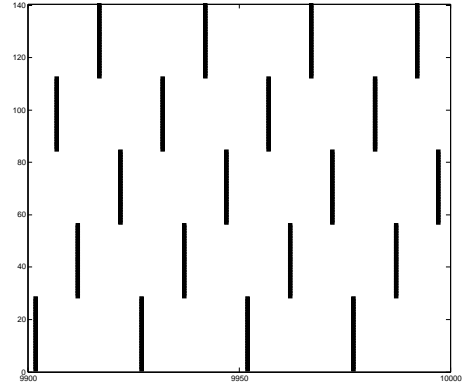}}
\subfigure[$\psi=6\pi/5$]{\includegraphics[width=2.5in]{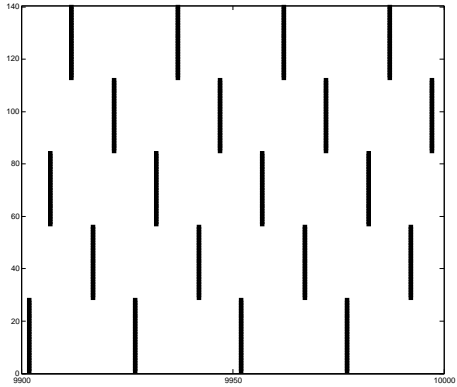}}
\subfigure[$\psi=8\pi/5$]{\includegraphics[width=2.5in]{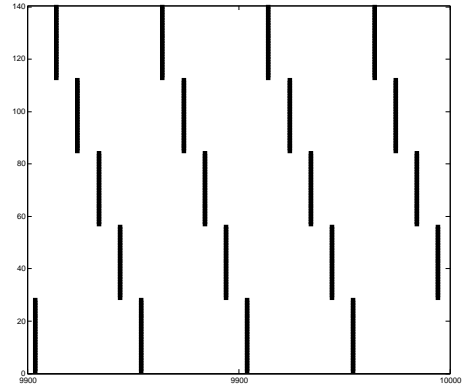}}
\caption{\small{\it Raster plots showing a stable 5-cluster solutions
in a network with $N=140$ neurons and bi-directional coupling
(connectivity matrix $W_1$). $\tau=12$ and $\epsilon=0.001$ all 
other parameters values are given in Table~\ref{table parameters}.
(a) $\psi=2\pi/5$, cluster ordering $C_1-C_2-C_3-C_4-C_5$
(b) $\psi=4\pi/5$, cluster ordering $C_1- C_4- C_2-C_5- C_3$
(c) $\psi=6\pi/5$, cluster ordering $C_1- C_3- C_5- C_2- C_4$
(d) $\psi=8\pi/5$, cluster ordering $C_1-C_5- C_4- C_3- C_2 $
}}
\label{figure 5C orders}
\end{figure}

Now consider the 7-cluster solution with connection matrix $W_1$. The phase model predicts that when $\tau = 13$ there exist six stable 7-cluster solutions with clusters:
\begin{eqnarray*}
C_1= \{1, 8, 15, \dots, 134 \},\\
C_2= \{2, 9, 16, \dots, 135 \},\\
\vdots \\
C_7= \{7, 14, 21, \dots, 140 \}.
\end{eqnarray*}

\begin{figure}[!htbp]\label{figure 7C orders}
\subfigure[$\psi=8\pi/7$]{\includegraphics[width=2.5in]{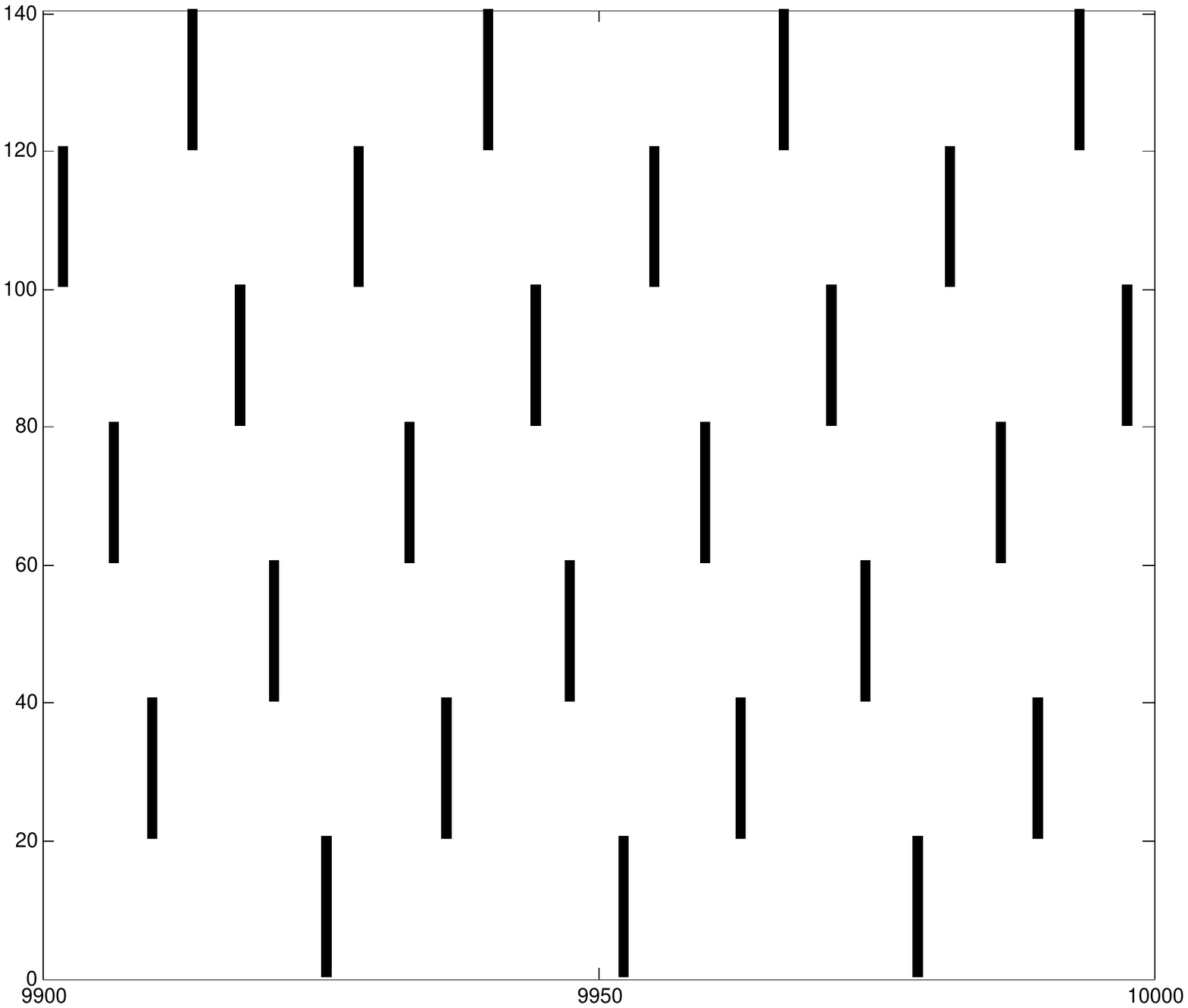}}
\subfigure[$\psi=6\pi/7$]{\includegraphics[width=2.5in]{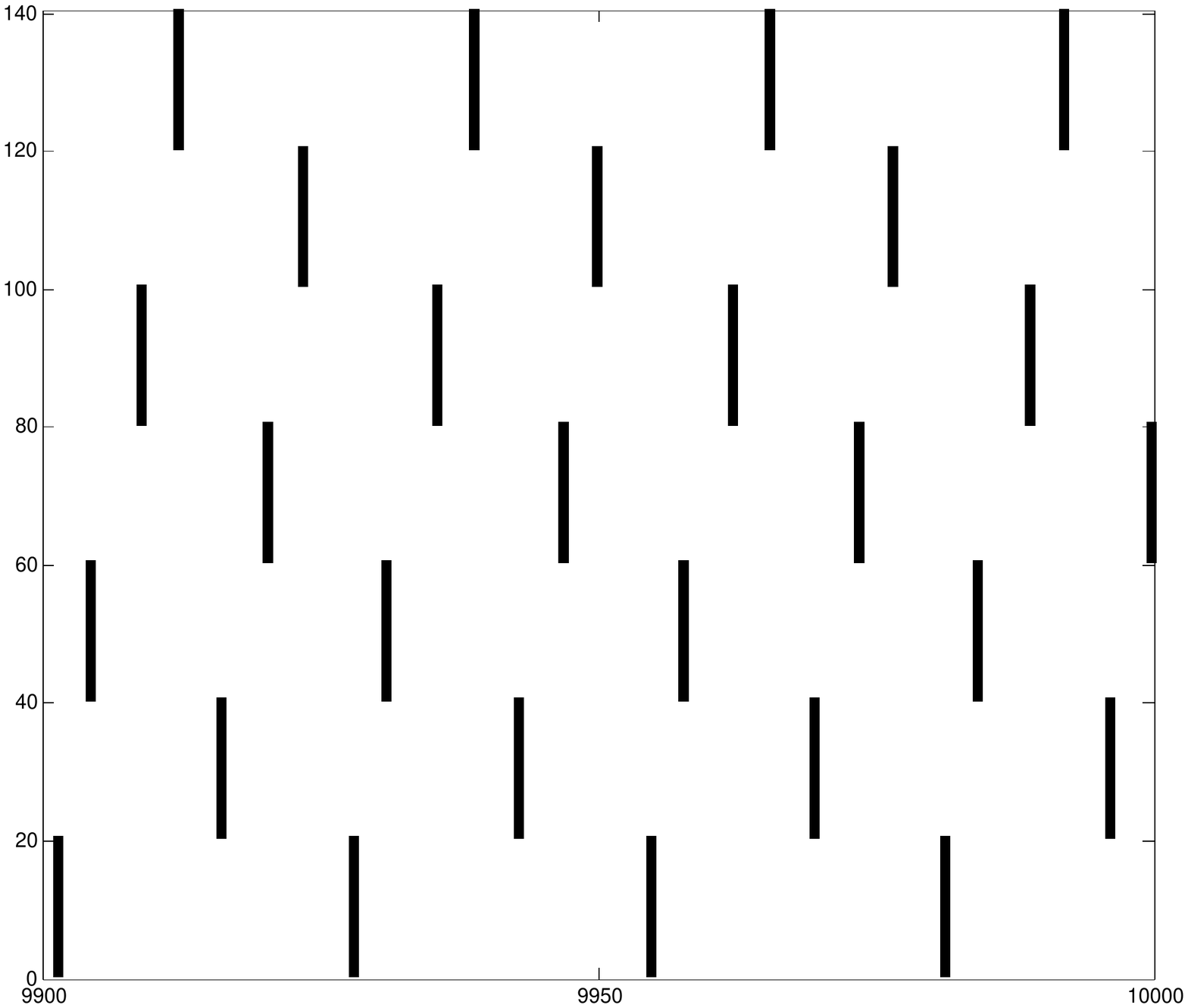}}
\caption{\small{\it Raster plots showing stable 7-cluster solutions
with $\tau=13$, $\epsilon=0.01$
in a network with $N=140$ neurons and bi-directional coupling (connectivity matrix $W_1$).
(a) $\psi= \frac{6\pi}{7}$, cluster ordering $C_1-C_6-C_4-C_7-C_5-C_3$.
(b) $\psi= \frac{8\pi}{7}$, cluster ordering $C_1-C_3-C_5-C_7-C_2-C_4-C_6$.
}}
\label{figure 7C orders}
\end{figure}

For $\psi = \frac{6\pi}{7}$, the cluster ordering is ${C_1-C_6- C_4-C_2- C_7- C_5- C_3}$
(see Figure \ref{figure 7C orders}(a)), while for $\psi = \frac{8\pi}{7}$, 
the cluster ordering is
${C_1-C_3- C_5- C_7- C_2-C_4-C_6}$ (see Figure \ref{figure 7C orders}(b)).  
In Figure \ref{figure 7C orders}, we reorder the oscillator indices so that oscillators that 
belong to the same cluster are plotted together. We were unable to find the other 
7-cluster solutions numerically.

\begin{remark}
We have observed other types of stable cluster solutions.
For example, Figure~\ref{figure AC2} shows solutions of the type \eqref{AC2 solution 1}
and \eqref{AC2 solution 2} which appear to be stable. With $N =8$ and bidirectional coupling in (\ref{ML define W1 W2}), the phase model predicts that the solutions of the type \eqref{AC2 solution 1} and \eqref{AC2 solution 2} are unstable for all $\tau$ when $\epsilon >0$, and stable for $\tau \in (1.5, 2.0] \cup (13.8,  14.1)$ when $\epsilon <0$. This prediction is consistent the 
numerically observed solution which occurs for $\epsilon = -0.01$, and $\tau = 2$.
\end{remark}

\begin{figure}[!htbp]
\subfigure[]{\includegraphics[width=3.0in]{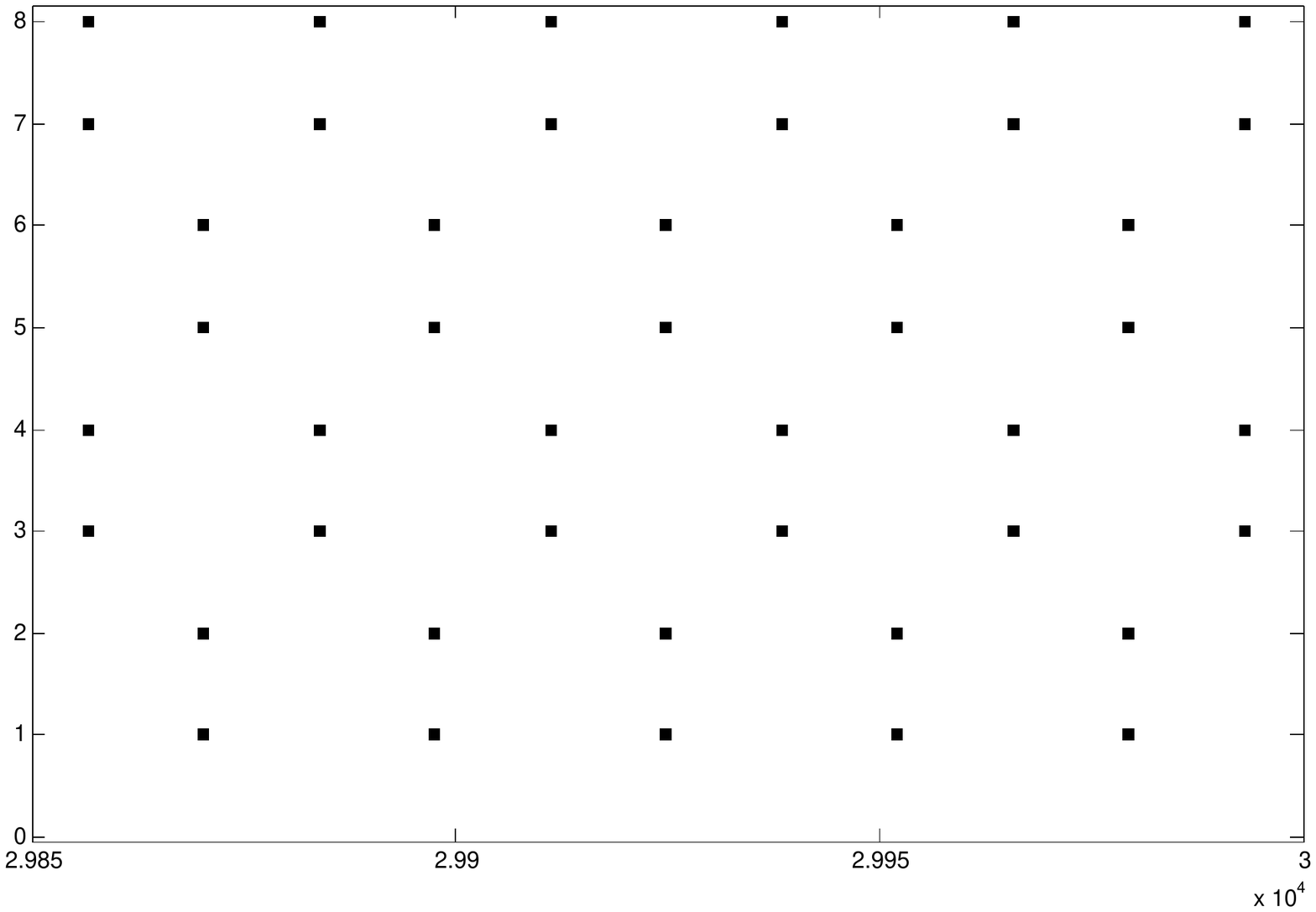}}
\subfigure[]{\includegraphics[width=3.0in]{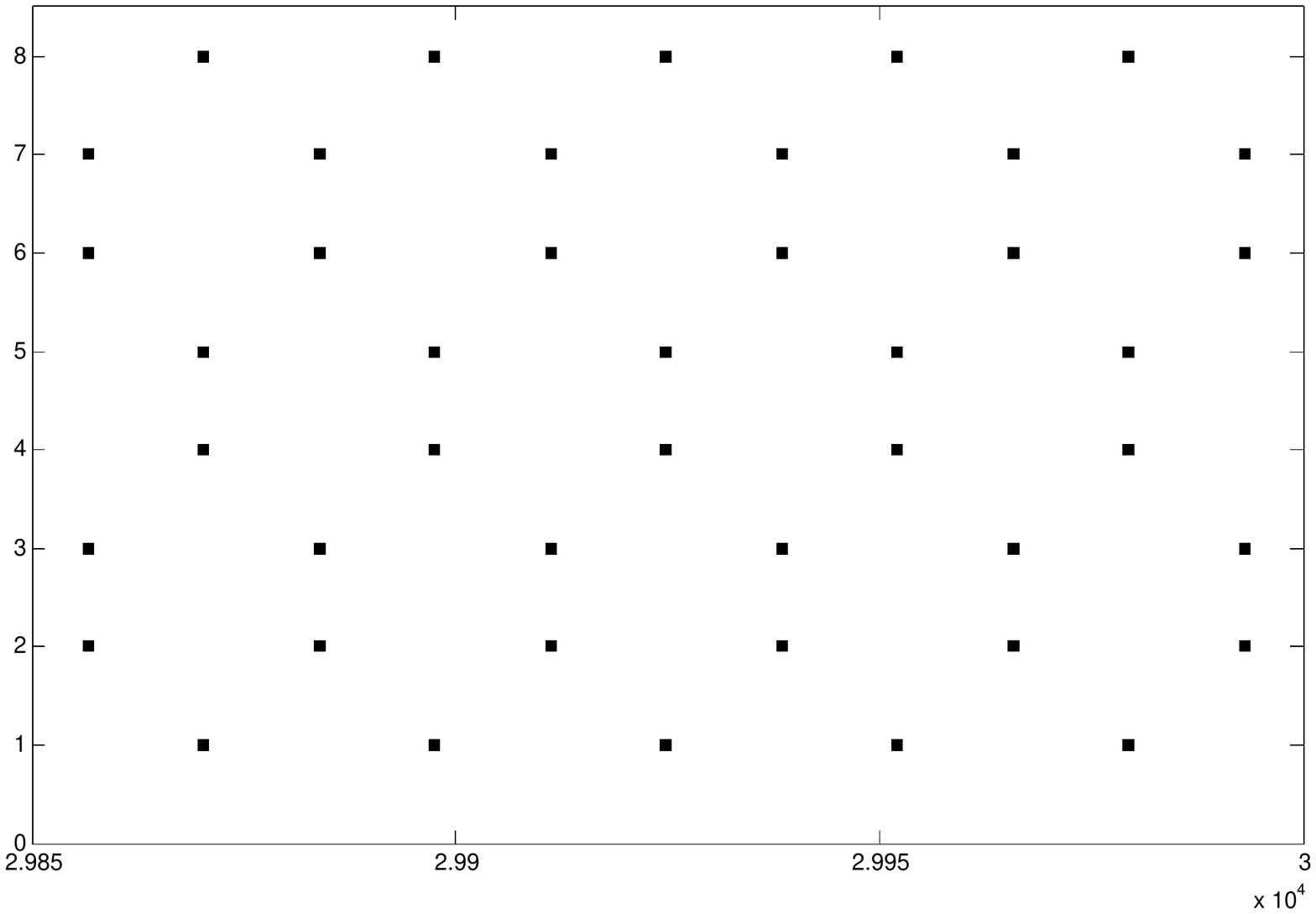}}
\caption{2-cluster solutions of the form (\ref{AC2 solution 1}) (a) and (\ref{AC2 solution 2}) (b) for $N=8$, $\epsilon=-0.01$, $\tau=2$ and connectivity matrix $W_1$. }\label{figure AC2}
\end{figure}

From Tables \ref{table W1} and \ref{table W2} it is clear that 
 the system exhibits multistability for a large of range of $\tau$ values. To further investigate the multistability, we carried out numerical simulations of the model (\ref{six Morris-Lecar model}) with $N=6$ and coupling matrix $W_1$ using XPPAUT \cite{Ermentrout02}. We start with a constant initial conditions ($v_i(t) = v_{i0}$, $w_i(t) = w_{i0}$, $-\tau \leq t \leq0$), and apply a small perturbation to the input current of one or more neurons during the simulation. The perturbations could cause switching between two different cluster types or between different realizations of the same cluster type. Figure \ref{figure 3C} show two examples, where the dark bars indicate when a particular neuron spikes. When $\tau =8$, both the 2-cluster solutions and 3-cluster solutions are stable. Figure \ref{figure 3C} (a) shows that when $\tau = 8$, a perturbation to neurons 1, 2, 3, 4 and 6 for $600 \leq t \leq 650$ switches the networks from a 3-cluster solution (with clusters (1, 4), (2, 5) and (3, 6)) to a 2-cluster solution (with clusters (1, 3, 5), and (2, 4, 6)). Figure \ref{figure 3C} (b) shows when $\tau = 8$, a perturbation to neuron 2, 4, 5, and 6 for $600 \leq t \leq 650$ switches the network from a 3-cluster solution with clusters  ordering (1, 4)-(3, 6)-(2, 5) to a 3-cluster solution with clusters ordering (1, 4)-(2, 5)-(3, 6).

\begin{figure}[!htbp]
\subfigure[$\tau=8$]{\includegraphics[scale=0.3]{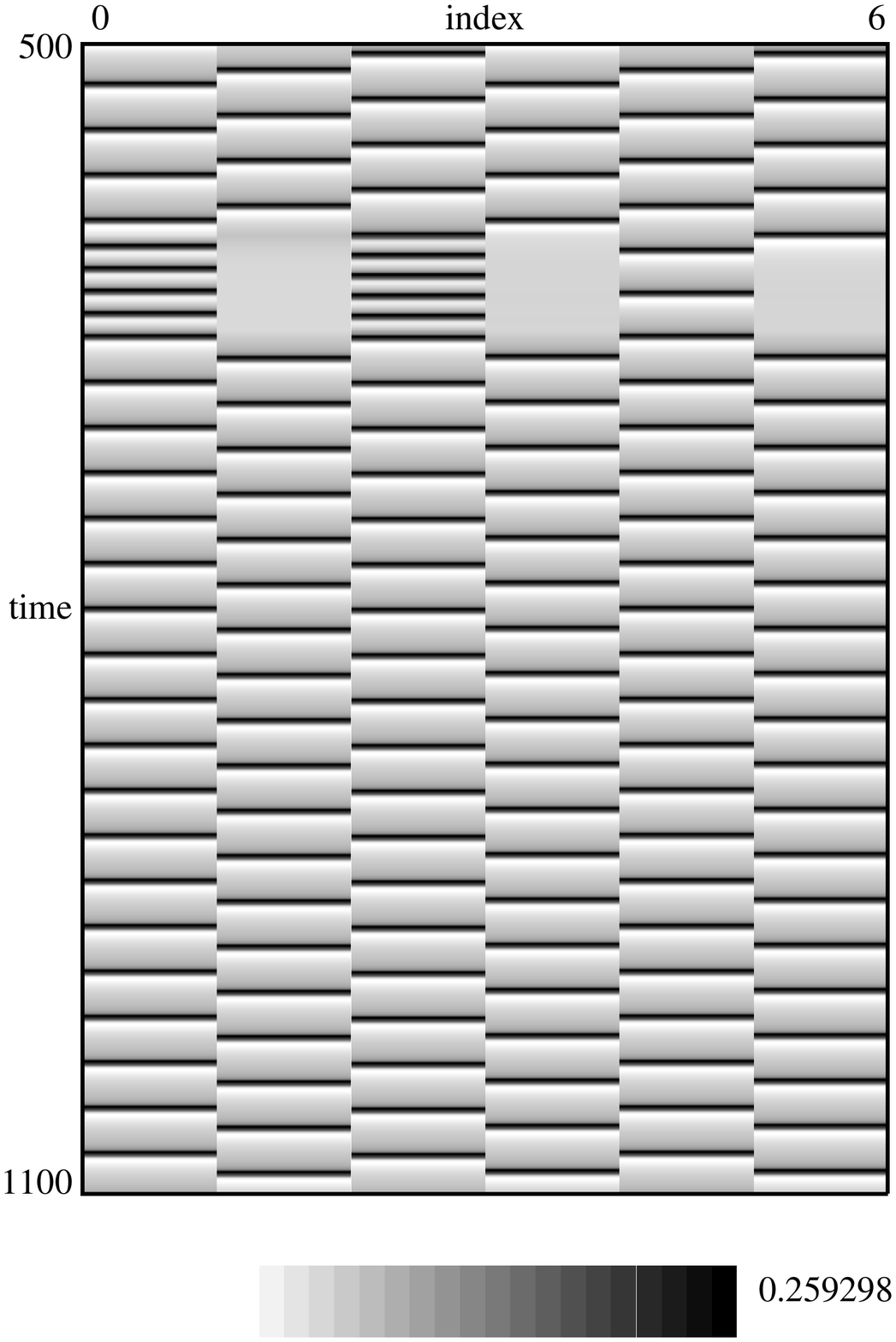}}
\subfigure[$\tau=8$]{\includegraphics[scale=0.3]{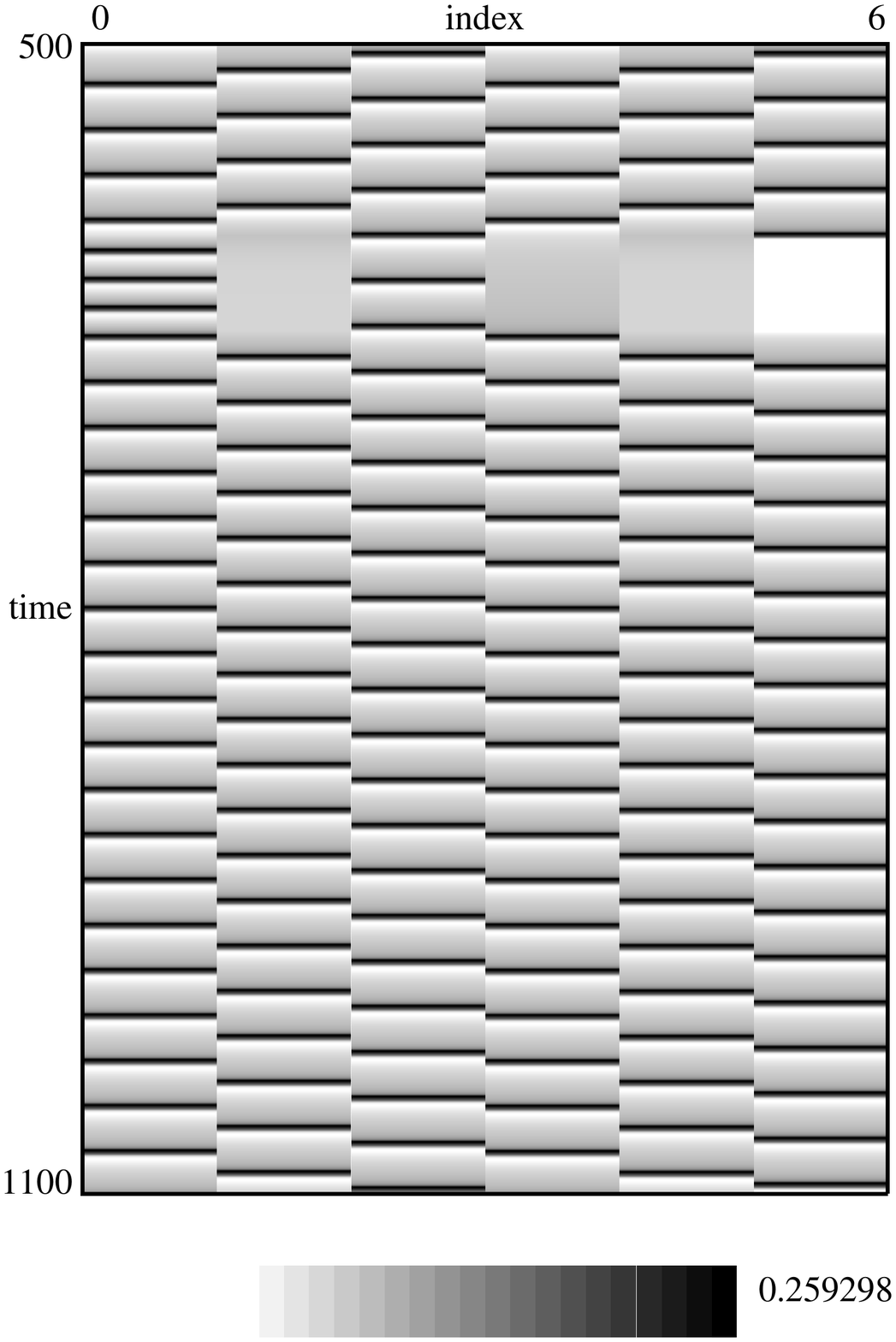}}
\caption{Numerical simulations showing multistability in a 6 neuron network
with bidirectional coupling (\ref{ML define W1 W2}). (a) Switching from a 3-cluster solution to a 2-cluster solution. (b) Switching from a 3-cluster solution to a 3-cluster solution.
$\tau = 8 $ and $\epsilon=0.001$. All other parameters are given in Table~\ref{table parameters}.
}
\label{figure 3C}
\end{figure}


\section{Persistence under symmetry breaking.}
\label{sec:pert}

By the weakly connected theory, the phase model analysis should persist under $\epsilon$-perturbation of the original model. From the steps of phase model reduction, we can see that if we perturb the connectivity matrix $W=(w_{ij})$ as $\tilde{W}= w_{ij}\,(1+\epsilon\, m_{ij})$,
the $\epsilon$-perturbation term will finally add to $\mathit{O}(\epsilon^2)$ term
in the phase model (\ref{DDE phase model 1}). A similar conclusion is obtained if we
perturb the time delay $\tau$ as $\tau_{ij} = \tau\, (1 + \epsilon\, \sigma_{ij})$.
Here $M = (m_{ij})$, and $S = (\sigma_{ij})$ are $N\times N$ matrices with elements which are $\mathit{O}(1)$ with respect to $\epsilon$. $\tau_{ij}$ represents transmission time from the $j$th oscillator to the $i$th oscillator.
Note that, after the perturbation, system (\ref{PM general network DDE}) no longer
possesses any symmetry. To $O(\epsilon)$ the symmetry persists, however. We thus
expect that, for $\epsilon$ sufficiently small, the analysis of
section~\ref{section phase difference analysis} should
still predict the behaviour of the system.

In order to investigate the effect of the $\epsilon$-perturbation on the connectivity matrix and time delay, we carried out sets of numerical simulations. For each set, we compare the original model with $W$ and $\tau$, to a model with $\tilde{W}$ and $\tau$, and a model with $W$ and $\tau_{ij}$. Take $N=6$, $W = circ\{0, 1, 1/2, 1/3, 1/2, 1 \}$, and $m_{ij}$, $\sigma_{ij}$ to be random numbers between 0 and 1. We simulate the original model and two perturbed models with $\tau = 1, \cdots, 15$, and $\epsilon = 0.001, 0.01, 0.05, 0.1$, respectively. From the simulation results, we see that for $\epsilon = 0.001, 0.01, 0.05$ the behavior of the perturbed models are the same as the unperturbed one for large time $t$. More accurately, the perturbed models take longer to settle at steady states than the original model. For $\epsilon = 0.1$, the behavior of unperturbed model almost captures the behavior of the perturbed ones. However, the system is sensitive to the $\tau$ values where steady states switch stability.  Therefore, we conclude that for a network with $6$ oscillators, the analysis of the original model is valid under perturbation with $\epsilon$ up to $0.05$. Furthermore, for a network with $N$ oscillators, the analysis of the system (\ref{PM general network DDE}) should persist under sufficiently small $\epsilon$-perturbation.

\begin{table}
\centering
\resizebox{1\textwidth}{!}{
\begin{tabular}{|c | c|c c c|c c c|c c c|c c c|}
\hline
\multirow{2}{*}{$\tau$ } & \multirow{2}{*}{PMP} & \multicolumn{3}{|c|}{$\epsilon = 0.001$} & \multicolumn{3}{|c|}{$\epsilon = 0.01$} & \multicolumn{3}{|c|}{$\epsilon = 0.05$}  & \multicolumn{3}{|c|}{$\epsilon = 0.1$} \\
\cline{3-14}
& & original & $\tilde{W}$ & $\tilde{\tau}$ & original & $\tilde{W}$ & $\tilde{\tau}$  & original & $\tilde{W}$ & $\tilde{\tau}$ & original & $\tilde{W}$ & $\tilde{\tau}$ \\
\hline
1 & 1C/3C  & 1C & 1C & 1C & NC & NC & NC & NC & NC & NC & 6C & NC & NC \\
2 & 3C  & 6C & 6C & 6C & 3C & 3C & 3C & 2C & 2C & 2C & 2C & 2C & 2C \\
3 & 2C/3C & 2C & 2C & 2C & 3C & 3C & 3C & 2C & 2C & 2C & 2C & 2C & 2C \\
\hline
4 & 2C/3C & 2C & 2C & 2C & 2C & 2C & 2C & 2C & 2C & 2C & 2C & 2C & 2C \\
5 & 2C/3C & 3C & 3C & 3C & 2C & 2C & 2C & 2C & 2C & 2C & 2C & 2C & 2C \\
6 & 2C/3C & 2C & 2C & 2C & 2C & 2C & 2C & 2C & 2C & 2C & 2C & 2C & NC \\
\hline
7 & 2C/3C & 2C & 2C & 2C & 2C & 2C & 2C & 2C & 2C & 2C & 1C & 1C & 1C  \\
8 & 2C/3C & 2C & 2C & 2C & 2C & 2C & 2C & 3C & 3C & NC & 1C & 1C & 1C  \\
9 & 2C/3C & 3C & 3C & 3C & 2C & 2C & 2C & NC & 1C & 1C & 1C & 1C & 1C   \\
\hline
10 & 3C & 3C & 3C & 3C & 1C & 1C & 1C & 1C & 1C & 1C & 1C & 1C & 1C\\
11 & 3C & NC & NC & NC & NC & NC & NC & 1C & 1C & 1C & 1C & 1C & 1C\\
12 & 3C & NC & NC & NC & NC & NC & NC & 1C & 1C & 1C & 1C & 1C & 1C\\
\hline
13 & 3C/6C & 6C & 6C & 6C & 1C & 1C & 1C & 1C & 1C & 1C & 1C & 1C & 1C\\
14 & 6C & 6C & 6C & 6C & 1C & 1C & 1C & 1C & 1C & 1C & 1C & 1C & 1C\\
15 & 1C & 1C & 1C & 1C & 1C & 1C & 1C & 1C & 1C & 1C & 1C & 2C & NC\\
\hline
\end{tabular}
}
\caption{Comparison of the original model and the two perturbed models for $\tau = 1, 2, \cdots, 15$ with $N=6$. The first column shows the stable cluster solutions predicted by the phase model for each $\tau$.}\label{table perturbation 3 models}
\end{table}

\section{Conclusions and future work}

In this paper, we studied a general system of identical oscillators with global
circulant, time-delayed coupling and showed that clustering behavior is a quite
prevalent
pattern of solution. We classified different clusters by the phase differences between neighboring oscillators, and investigated the existence and linear stability of
clustering solutions. We focussed on symmetric cluster solutions, where the
same number of oscillators belong to each cluster.  In particular, we showed that
certain symmetric cluster solutions exist for any type of oscillator
and any value of the delay -- their existence depends only on the presence of
circulant coupling.
We gave a complete analysis of the linear stability of these cluster
solutions. In the case of global bidirectional coupling and global homogeneous
coupling, more details about how the stability changes with parameters could be obtained using the symmetry.

Further exploration was done through numerical continuation and numerical
simulation studies of a specific example: circulantly coupled Morris-Lecar
oscillators. We considered both small ($N=6,8$) and large ($N=140$) networks
and two types of coupling: homogeneous and bi-directional, distance dependent.
As expected, the numerical studies agree with the theoretical predictions of the
phase model, so long as the strength of the coupling ($\epsilon$)
was sufficiently small.
For the parameters we explored this was $\epsilon\lesssim 0.05$.
In all cases we explored,
the $1-$cluster (synchronous) solution was the only asymptotically stable
solution when there was no delay in the system. For non-zero delay, this
solution could become unstable and other cluster solutions became stable.
We found ranges of the delay for which the system exhibits a high degree of
multistability. The multistability persisted even under in perturbations of
the coupling matrix ($W$), and time delay ($\tau$) which break the symmetry
of the model. The perturbed model agreed with the phase model prediction for
$\epsilon\lesssim 0.01$.

Delay-induced multistability has been observed in Hopfield neural networks
(e.g., \cite{ma2009,Zou2012}), in networks of spiking neurons
\cite{Ma2007,Foss1996,Foss1997}, and even in experimental
systems \cite{Foss2000}, where it has been postulated
as a potential mechanism for memory storage. The multistability we observe
has similar potential. It also provides the network with a simple way to
respond differently to different inputs, without changing synaptic weights.
Switching between solutions with a different number of clusters changes
the network average frequency, which could then change how the network affects
downstream neurons.

Multistability between different cluster solutions also has potential
connections with the concept of neural assemblies.  A neural assembly is
a group of neurons which transiently act together to achieve a particular
purpose \cite{dragoi2006,harris2003,pastalkova2008}. A network with multiple stable
cluster solutions provides a basic model for such behaviour. As the system switches
between different cluster solutions different neurons become synchronized
with each other. As we have shown, it possible for network to possess
multiple stable solutions with the same number of clusters but with different
groupings of the neurons.

In the future, it would be interesting to pursue a variety of the directions suggested by our results.
The switching of stability of the cluster solutions as the delay is varied should be
associated with bifurcations in the model. In the case of system with two neurons
it has been shown that delay induced stability changes of the $1-$ and $2-$ cluster
solutions are associated with pitchfork and saddle-node bifurcations in the phase
model and sometimes involve
other phase-locked solutions \cite{CK12}.  It would be interesting to explore the
delay induced bifurcations that occur in our network model. Preliminary numerical
investigations of the phase model (not shown) indicate a quite complex bifurcation
structure.  It would also be interesting to compute the bifurcation structure of the
cluster solutions in the $(\tau, \epsilon)$ parameter plane to get a better
understanding of the limits of the validity of the phase model.

\bibliography{/home/sacampbe/tex/bibs/abbrev,refs}







\end{document}